\documentclass[11pt,twoside]{amsart}
\usepackage{latexsym,amssymb,amsmath}

\numberwithin{equation}{section}
\newtheorem{theorem}{Theorem}[section]
\newtheorem{lemma}[theorem]{Lemma}
\newtheorem{proposition}[theorem]{Proposition}
\newtheorem{corollary}[theorem]{Corollary}

\theoremstyle{definition}
\newtheorem{definition}[theorem]{Definition} 
\newtheorem{procedure}[theorem]{Procedure} 
\newtheorem{remark}[theorem]{Remark}
\newtheorem{example}[theorem]{Example}

\begin{document}


\title[]{Parameterized affine codes}

\author{Hiram H. L\'opez}
\address{
Departamento de
Matem\'aticas\\
Centro de Investigaci\'on y de Estudios
Avanzados del
IPN\\
Apartado Postal
14--740 \\
07000 Mexico City, D.F.
}
\email{hlopez@math.cinvestav.mx}

\author{Eliseo Sarmiento}
\address{
Departamento de
Matem\'aticas\\
Centro de Investigaci\'on y de Estudios
Avanzados del
IPN\\
Apartado Postal
14--740 \\
07000 Mexico City, D.F.
}
\email{esarmiento@math.cinvestav.mx}

\thanks{The second author was partially supported by CONACyT. The
third author is a member of the Center for Mathematical Analysis,
Geometry and Dynamical Systems. The fourth author was partially
supported by SNI}

\author{Maria Vaz Pinto}
\address{
Departamento de Matem\'atica\\
Instituto Superior Tecnico\\
Universidade T\'ecnica de Lisboa\\ 
Avenida Rovisco Pais, 1\\ 
1049-001 Lisboa, Portugal 
}\email{vazpinto@math.ist.utl.pt}

\author{Rafael H. Villarreal}
\address{
Departamento de
Matem\'aticas\\
Centro de Investigaci\'on y de Estudios
Avanzados del
IPN\\
Apartado Postal
14--740 \\
07000 Mexico City, D.F.
}
\email{vila@math.cinvestav.mx}


\begin{abstract} Let $K$ be a finite field and let $X^*$ be an affine
algebraic toric set parameterized by monomials. We give
an algebraic method, using Gr\"obner bases, 
to compute the length and the dimension of
$C_{X^*}(d)$, the parameterized affine code of degree $d$ on 
the set $X^*$. If $Y$ is the projective closure of $X^*$, it is shown
that $C_{X^*}(d)$ has the same basic parameters that $C_Y(d)$, the
parameterized projective code on the set $Y$.   
If $X^*$ is an affine torus, we compute the basic
parameters of $C_{X^*}(d)$. We show how to compute the vanishing
ideals of $X^*$ and $Y$. 
\end{abstract}

\maketitle

\section{Introduction}\label{intro-affine-codes}

Let $K=\mathbb{F}_q$  be a finite field with $q$ elements and 
let $y^{v_1},\ldots,y^{v_s}$ be a finite set of monomials.  
As usual if $v_i=(v_{i1},\ldots,v_{in})\in\mathbb{N}^n$, 
then we set 
$$
y^{v_i}=y_1^{v_{i1}}\cdots y_n^{v_{in}},\ \ \ \ i=1,\ldots,s,
$$
where $y_1,\ldots,y_n$ are the indeterminates of a ring of 
polynomials with coefficients in $K$. Consider the following set
parameterized  by these monomials 
$$
X^*:=\{(x_1^{v_{11}}\cdots x_n^{v_{1n}},\ldots,x_1^{v_{s1}}\cdots
x_n^{v_{sn}})\in\mathbb{A}^{s}	\vert\, x_i\in K^*\mbox{ for all }i\},
$$
where $K^*=K\setminus\{0\}$ and $\mathbb{A}^{s}=K^s$ is an affine
space over the field $K$. We call $X^*$ an 
{\it affine algebraic toric set\/} parameterized  by 
$y^{v_1},\ldots,y^{v_s}$. The set $X^*$ is a multiplicative group under
componentwise multiplication. 

Let
$S=K[t_1,\ldots,t_s]=\oplus_{d=0}^\infty S_d$ 
be a polynomial ring 
over the field $K$ with the standard grading, let $P_1,\ldots,P_m$
be the points of $X^*$, and let $S_{\leq d}$ be the set of polynomials of $S$ of
degree at most $d$. The {\it evaluation map\/} 
\begin{equation}\label{ev-map}
{\rm ev}_d\colon S_{\leq d}\longrightarrow K^{|X^*|},\ \ \ \ \ 
f\mapsto \left(f(P_1),\ldots,f(P_m)\right),
\end{equation}
defines a linear map of
$K$-vector spaces. The image of ${\rm ev}_d$, denoted by $C_{X^*}(d)$,
defines a {\it linear code}. We call
$C_{X^*}(d)$ a {\it parameterized affine code\/} of
degree $d$ on the set $X^*$. As usual by a {\it linear code\/} we mean a linear subspace of
$K^{|X^*|}$. Parameterized affine codes are special types of
affine Reed-Muller codes in the sense of \cite[p.~37]{tsfasman}. If $s=n=1$ and $v_1=1$, then
$X^*=\mathbb{F}_q^*$ and we obtain the classical Reed-Solomon code of
degree $d$ \cite[p.~42]{stichtenoth}. 

The {\it dimension\/} and the {\it length\/} of $C_{X^*}(d)$ 
are given by $\dim_K C_{X^*}(d)$ and $|{X^*}|$ respectively. The dimension
and the length 
are two of the {\it basic parameters} of a linear code. A third
basic parameter is the {\it minimum
distance\/} which is given by 
$$\delta_{X^*}(d)=\min\{\|v\|
\colon 0\neq v\in C_{X^*}(d)\},$$ 
where $\|v\|$ is the number of non-zero
entries of $v$.  

The basic parameters of $C_{X^*}(d)$ are related by the
{\it Singleton bound\/} for the minimum distance
\begin{equation}\label{singleton-bound}
\delta_{X^*}(d)\leq |{X^*}|-\dim_KC_{X^*}(d)+1.
\end{equation}

The contents of this paper are as follows. Let $\mathbb{P}^s$ be the
projective space over the field $K$. In 
Theorem~\ref{bridge-affine-projective}, it is shown that $C_{X^*}(d)$ has
the same parameters that $C_Y(d)$, the
parameterized projective code of degree $d$ on $Y$ (see
Definition~\ref{pac}), where $Y$ is the image of $X^*$ under the map 
$\mathbb{A}^{s}\rightarrow \mathbb{P}^{s} $, $x\mapsto [(x,1)]$. It is also shown that the dimension and
the length of a parameterized affine code can be expressed in terms of
the Hilbert function and the degree of the vanishing ideal 
$I(Y)$ of $Y$. 

As an application, if $T$ is an affine torus we
compute the basic parameters of $C_{T}(d)$ (see
Corollaries~\ref{may3-11} and \ref{may3-1-11}). The basic parameters of
other types of Reed-Muller codes (or evaluation codes) over finite
fields have been computed in 
a number of cases. If $X=\mathbb{P}^{s}$,  
the parameters of $C_X(d)$ are described in
\cite[Theorem~1]{sorensen}. If $X$ is the image of $\mathbb{A}^s$ under the map 
$\mathbb{A}^{s}\rightarrow \mathbb{P}^{s} $, $x\mapsto [(x,1)]$, the parameters 
of $C_X(d)$ are described in
\cite[Theorem~2.6.2]{delsarte-goethals-macwilliams}. If $X\subset\mathbb{P}^s$ is
a set parameterized by monomials arising from the edges of a clutter
and the vanishing ideal of $X$ is a complete intersection, the parameters of $C_X(d)$ are
described in \cite{ci-codes}.   

In Theorem~\ref{22-06-10}, we show how to compute the vanishing ideal
of $X^*$. Then, we show how to compute the vanishing 
ideal of $Y$ using
Gr\"obner bases (see Lemma~\ref{elim-ord-hhomog}). We
obtain a method to compute the dimension and the length of
$C_{X^*}(d)$ using the computer algebra system {\it
Macaulay\/}$2$ \cite{mac2} (see Corollary~\ref{may4-11} and
Procedure~\ref{mac-proc-affine-codes}).  

For all unexplained
terminology and additional information  we refer to
\cite{CLO,Sta1,Stur1} 
(for the theory of Gr\"obner bases, Hilbert functions, and toric
ideals),  
\cite{MacWilliams-Sloane,stichtenoth,tsfasman} (for the theory of 
linear codes), and
\cite{gold-little-schenck,GR,GRH,GRT,renteria-tapia-ca2} for
the theory of Reed-Muller codes and evaluation codes. 

\section{Computing the length and dimension of an affine parameterized
code}\label{computing-l-d}

We continue to use the notation and definitions used in the
introduction. In this section we study parameterized affine codes and
show how to express its dimension and length in terms of the Hilbert
function and the degree of a certain standard graded algebra.  

Let $\mathbb{P}^{s}$ be a projective space over the field $K$. Consider the algebraic toric set
$$
Y:=\{[(x_1^{v_{11}}\cdots x_n^{v_{1n}},\ldots,x_1^{v_{s1}}\cdots
x_n^{v_{sn}},1)]\, \vert\, x_i\in K^*\mbox{ for all
}i\}\subset\mathbb{P}^{s},
$$
where $K^*=K\setminus\{0\}$. Notice that  $Y$ is parameterized  by
$y^{v_1},\ldots,y^{v_s},y^{v_{s+1}}$, where $v_{s+1}=0$. Also notice
that $Y$ is the projective closure of $X^*$ because $K$ is a finite field (see
Section~\ref{section-computing-ld}). The sets $X^*$ and $Y$ have the
same cardinality because the map $\rho\colon X^*\rightarrow Y$,
$x\mapsto [(x,1)]$, is bijective.  

The {\it vanishing ideal\/}
of $Y$, denoted by $I(Y)$, is the ideal  
of $S[u]$ generated by the homogeneous polynomials that 
vanish on $Y$, where $u=t_{s+1}$ is a new variable and 
$S[u]=\oplus_{d\geq 0}S[u]_d$ is a polynomial
ring, with the standard grading, 
over the field $K$. Let $Q_1,\ldots,Q_m$ be a set of 
representatives
for the points of $Y$ and let $f_0(t_1,\ldots,t_{s+1})=t_1^d$. The evaluation map
$$
{\rm ev}'_d\colon S[u]_d\longrightarrow K^{|Y|},\ \ \ \ \ 
f\mapsto
\left(\frac{f(Q_1)}{f_0(Q_1)},\ldots,\frac{f(Q_m)}{f_0(Q_m)}\right),
$$
defines a linear map of $K$-vector spaces. If $Q_1',\ldots,Q_m'$ is
another set of representatives, then there are
$\lambda_1,\ldots,\lambda_m$ in $K^*$ such that $Q_i'=\lambda_iQ_i$
for all $i$. Thus, $f(Q_i')/f_0(Q_i')=f(Q_i)/f_0(Q_i)$ for $f\in
S[u]_d$ and $1\leq i\leq m$. This means that the map ${\rm ev}'_d$ is
independent of
the set of representatives that we choose for the points of $Y$. In
what follows we choose $(P_1,1),\ldots,(P_m,1)$ as a set of
representatives for the points of $Y$. 

\begin{definition}\label{pac}
The image of ${\rm ev}'_d$, denoted by $C_Y(d)$,
defines a {\it linear code} that we call a {\it parameterized
projective code\/} of degree $d$.
\end{definition}

\begin{definition} The {\it Hilbert
function\/} of $S[u]/I(Y)$ is given by 
$$H_Y(d):=
\dim_K(S[u]_d/I({Y})\cap S[u]_d),$$
and the {\it Krull-dimension\/} of $S[u]/I(Y)$ is denoted by
$\dim(S[u]/I(Y))$.
\end{definition}

The unique polynomial $h_Y(t)=\sum_{i=0}^{k-1}c_it^i\in
\mathbb{Z}[t]$ of degree $k-1=\dim(S[u]/I(Y))-1$ such that
$h_Y(d)=H_Y(d)$ for 
$d\gg 0$ is called the {\it Hilbert polynomial\/} of $S[u]/I(Y)$, see
\cite{Sta1}. The
integer $c_{k-1}(k-1)!$, denoted by ${\rm deg}(S[u]/I(Y))$, is 
called the {\it degree\/} or  {\it multiplicity} of $S[u]/I(Y)$.

\begin{proposition}{\rm(\cite[Lecture 13]{harris},
\cite{geramita-cayley-bacharach})}\label{harris-geramita} $h_Y(d)=|Y|$ for $d\geq |Y|-1$.
\end{proposition} 

Recall that the {\it vanishing ideal} of ${X^*}$, denoted by $I({X^*})$, consists
of all polynomials $f$ of $S$ that vanish on the set ${X^*}$. Given
$f\in S_{\leq d}$, we set
$$
f^\mathfrak{h}(t_1,\ldots,t_s,u):=u^df(t_1/u,\ldots,t_s/u).$$  
The polynomial $f^\mathfrak{h}$ is homogeneous of degree $d$. The
polynomial $f^\mathfrak{h}$ is called the {\it homogenization} of $f$ with respect
to $u$ and $d$.

\begin{theorem}\label{bridge-affine-projective} {\rm(a)} There is
an isomorphism of $K$-vector spaces
$\varphi\colon C_{X^*}(d)\rightarrow C_Y(d)$, 
$$
(f(P_1),\ldots,f(P_m))\stackrel{\varphi}{\longmapsto}
\left(\frac{f^\mathfrak{h}(P_1,1)}{f_0(P_1,1)},\ldots,
\frac{f^\mathfrak{h}(P_m,1)}{f_0(P_m,1)}\right)=
\left(\frac{f(P_1)}{f_0(P_1)},\ldots,
\frac{f(P_m)}{f_0(P_m)}\right).
$$

{\rm(b)} The parameterized codes $C_{X^*}(d)$ and $C_Y(d)$ have the
same parameters. 

{\rm(c)} The dimension and the length of $C_{X^*}(d)$ are 
$H_Y(d)$ and ${\rm deg}(S[u]/I(Y))$ respectively.
\end{theorem}

\begin{proof} (a) We set $I({X^*})_{\leq d}=I({X^*})\cap S_{\leq d}$.
The kernel of ${\rm ev}_d$ is precisely $I({X^*})_{\leq 
d}$. Hence, there is an isomorphism of $K$-vector spaces
\begin{equation}\label{isom-rs-1}
S_{\leq d}/I({X^*})_{\leq d}\simeq
C_d({X^*})=\{\left(f(P_1),\ldots,f(P_m)\right)\vert\, f\in S_{\leq
d}\}.
\end{equation}

 The kernel of ${\rm ev}'_d$ is the homogeneous part
$I(Y)_d$ of degree $d$ of $I(Y)$. Notice that $I(Y)_d$ is equal to 
$I({Y})\cap S[u]_d$. Therefore, there is an isomorphism of $K$-vector spaces
\begin{equation}\label{isom-rs-2}
S[u]_d/I(Y)_d\simeq C_Y(d).
\end{equation}

The homogenization map $\psi\colon S_{\leq d}{\rightarrow}S[u]_d$, $f\mapsto
f^\mathfrak{h}$,  is an isomorphism of $K$-vector spaces (see
\cite[p.~330]{singular}) such that
$\psi(I(X^*)_{\leq d})=I(Y)_d$. Hence, the induced map
\begin{equation}\label{may5-11}
\Phi\colon S_{\leq d}\rightarrow S[u]_d/I(Y)_d,\ \ \ \ \ f\longmapsto
f^\mathfrak{h}+I(Y)_d,
\end{equation}
is a surjection. Thus, by Eqs.~(\ref{isom-rs-1}) and
(\ref{isom-rs-2}), it suffices to
observe that ${\rm ker}(\Phi)=I(X^*)_{\leq d}$.

(b) From part (a) it is clear that $C_{X^*}(d)$ and $C_Y(d)$ have the
same dimension and length. To show that they have the same minimum
distance it suffices to notice that the isomorphism $\varphi$ between
$C_{X^*}(d)$ and $C_Y(d)$ preserves the norm, i.e.,
$\|v\|=\|\varphi(v)\|$ for $v\in C_{X^*}(d)$. 

(c) The ring $S[u]/I(Y)$ has Krull-dimension $1$ (see
\cite[Theorem~2.1(c), p.~85]{algcodes}), thus its Hilbert polynomial
$h_Y(t)=c_0$ is a non-zero constant and its degree is equal to $c_0$.
Then, by Proposition~\ref{harris-geramita}, we get
$$
|Y|=h_Y(d)=c_0=\deg(S[u]/I(Y))
$$
for $d\geq |Y|-1$. Thus, $|Y|$ is 
the degree of $S[u]/I(Y)$. Hence, from part (b), we get that the
length of $C_{X^*}(d)$ is equal to the degree of $S[u]/I(Y)$ and the
dimension of $C_{X^*}(d)$ is equal to $H_Y(d)$. 
\end{proof}

From this result it follows at once that the codes $C_{X^*}(d)$ and
$C_Y(d)$ are equivalent in the sense of \cite[p.~48]{stichtenoth}.

\begin{remark} If $H_{X^*}(d)$ is the {\it affine Hilbert function\/} 
of the affine $K$-algebra $S/I(X^*)$, given by
$$H_{X^*}(d):=\dim_K\, S_{\leq d}/I(X^*)_{\leq d},$$ 
then, by Eq.~(\ref{may5-11}), $H_Y(d)=H_{X^*}(d)$ for $d\geq 1$ (see
\cite[Remark~5.3.16]{singular}).
\end{remark}

\begin{corollary} {\rm(a)} The dimension of $C_{X^*}(d)$ is increasing, as a function of 
$d$, until it reaches a constant value equal to $|X^*|$. {\rm(b)} The minimum distance of
$C_{X^*}(d)$ is decreasing, as a function of $d$, until
it reaches a constant value equal to $1$.
\end{corollary}

\begin{proof} The dimension of $C_Y(d)$ is increasing, as a function of 
$d$, until it reaches a constant value equal to $|Y|$ 
(see \cite[Remark~1.1, p.~166]{geramita-cayley-bacharach} or
\cite[p.~456]{duursma-renteria-tapia}). The minimum distance of $C_Y(d)$ is decreasing, as a 
function of $d$, until it reaches a constant value equal to $1$. This was shown in
\cite[Proposition~5.1, p.~99]{algcodes} and
\cite[Proposition~2.1]{tohaneanu}. 
Therefore the result follows from
Theorem~\ref{bridge-affine-projective}.  
\end{proof}

Next, we give an application by computing the basic parameters of a
certain family of parameterized affine codes. Let $X^*$ be an affine
algebraic toric set parameterized by $y_1,\ldots,y_s$. In this case we
denote $X^*$ by $T$ and $Y$ by $\mathbb{T}$. We call $T$ (resp.
$\mathbb{T}$) an {\it affine\/} (resp. {\it projective\/}) {\it
torus\/}. Recall that $T$ and $\mathbb{T}$ are given by  
$$
T=\left\{\left(x_1,\ldots,x_s\right)\mid x_i\in
K^\ast\right\}\subset\mathbb{A}^s\ \mbox{ and }\ 
\mathbb{T}=\left\{\left[(x_1,\ldots,x_s,1\right)]\mid x_i\in
K^\ast\right\}\subset\mathbb{P}^s,
$$
respectively.

\begin{corollary}\label{may3-11} The minimum distance of
$C_T(d)$ is given by
$$
\delta_T(d)=\left\{\begin{array}{cll}
(q-1)^{s-k-1}(q-1-\ell)&\mbox{if}&d\leq (q-2)s-1,\\
1&\mbox{if}&d\geq (q-2)s,
\end{array}
 \right.
$$
where $k$ and $\ell$ are the unique integers such that $k\geq 0$,
$1\leq \ell\leq q-2$ and $d=k(q-2)+\ell$. 
\end{corollary}

\begin{proof} It was shown in \cite{ci-codes} that the minimum distance of
$C_\mathbb{T}(d)$ is given by the formula above. Thus, 
by Theorem~\ref{bridge-affine-projective}, the result follows.
\end{proof}

A linear code is called
{\it maximum distance  separable\/} (MDS for short) if equality holds
in the Singleton bound (see Eq.~(\ref{singleton-bound})). As a 
consequence of this result we obtain the well-known formula
for the minimum distance of a Reed-Solomon code
\cite[p.~42]{stichtenoth}. 

\begin{corollary}{\rm(Reed-Solomon codes)} Let $T$ be an affine torus in $\mathbb{A}^1$. 
Then the minimum distance $\delta_{T}(d)$ of $C_{{T}}(d)$ is given by
$$
\delta_{T}(d)=\left\{\hspace{-1mm}\begin{array}{cll}
q-1-d&\mbox{if}&1\leq d\leq q-3,\\
1&\mbox{if}&d\geq q-2,
\end{array}
 \right.
$$
and $C_{T}(d)$ is an MDS code.
\end{corollary}

\begin{proof} In this situation $s=1$. If $d\leq q-3$, we can write
$d=k(q-2)+\ell$, where $k=0$ and $\ell=d$. Then, by
Corollary~\ref{may3-11}, we get $\delta_T(d)=q-1-d$ for $d\leq q-3$
and $\delta_T(d)=1$ for $d\geq q-2$.
\end{proof}

\begin{corollary}\label{may3-1-11} The length of $C_T(d)$ is $(q-1)^{s}$ and its dimension is 
$$
\dim_KC_T(d)=\sum_{j=0}^{\left\lfloor\frac{d}{q-1}\right\rfloor}(-1)^j{s\choose
j}{s+d-j(q-1)\choose s}.
$$
\end{corollary}

\begin{proof} The length of $C_T(d)$ is clearly equal to $(q-1)^{s}$
because $T=(K^*)^s$. It was shown in \cite{duursma-renteria-tapia} that the
dimension of $C_{\mathbb{T}}(d)$ is given by the formula above. 
Thus, by Theorem~\ref{bridge-affine-projective}, the result follows.
\end{proof}

\begin{example}\label{large-minimum-distance-1} 
Let $T$ be an affine torus in $\mathbb{A}^2$ and let $C_T(d)$ be its 
parameterized affine code of degree $d$ over the field
$K=\mathbb{F}_{11}$. Using Corollaries~\ref{may3-11} and
\ref{may3-1-11}, we obtain:

\begin{eqnarray*}
&&\left.
\begin{array}{c|c|c|c|c|c|c|c|c|c|c|c|c|c}
 d & 1 & 2 & 3 & 4 & 5 & 6 & 7 & 8 & 9 & 10 & 11 & 12 & 13  \\
   \hline
 |T| & 100 & 100 & 100 & 100 & 100 & 100 & 100 & 100 & 100 & 100 & 100 & 100 & 100 \\
   \hline
 \dim C_T(d)    \    & 3 & 6   & 10 & 15 & 21 & 28 & 36 & 45 & 55 & 64 & 72 & 79 & 85  \\
   \hline
 \delta_{T}(d) & 90 & 80 & 70 & 60 & 50 & 40 & 30 & 20 & 10 & 9  & 8  & 7  & 6   \\
\end{array}
\right.
\end{eqnarray*}
\end{example}

\section{Computing the dimension and length of
$C_{X^*}(d)$}\label{section-computing-ld}

We continue to use the notation and definitions used in
Sections~\ref{intro-affine-codes} and \ref{computing-l-d}. 
In this section we give expressions for $I(X^*)$ and $I(Y)$---valid over any
finite field $K$ with $q$ elements---that allow to compute some of the
basic parameters of a parameterized affine code using Gr\"obner bases. 

\begin{theorem}{\rm(Combinatorial Nullstellensatz
\cite[Theorem~1.2]{alon-cn})}
\label{comb-null} Let $R=K[y_1,\ldots,y_n]$ be a
polynomial ring over a field $K$, let $f\in R$, and let
$a=(a_i)\in\mathbb{N}^n$. Suppose that the coefficient of
$y^a$ in $f$ is non-zero and $\deg\left(f\right)=a_1+\cdots+a_n$. If
$S_{1},\ldots ,S_{n}$ are subsets of $K$, with $\left|S_{i}\right| > a_i$ for
all $i$, then there are $s_{1}\in S_{1},\ldots,s_n\in S_n$ such that
$f\left(s_{1},\ldots ,s_{n}\right) \neq 0$.  
\end{theorem}

\begin{lemma}\label{may1-2-11} Let $K=\mathbb{F}_{q}$ and let
$G$ be a polynomial in $K[ y_{1},\ldots,y_{n}] $. If $G$
vanishes on $\left( K^{*}\right)^{n}$ and $\deg_{y_{i}}\left(G\right)
< q-1$ for $i=1,\ldots ,n$, then $G=0$.
\end{lemma}
\begin{proof} We proceed by contradiction. Assume that $G$ is
non-zero. Then, there is a monomial $y^a$ that occurs in $G$ with 
$\deg(G)=a_1+\cdots+a_n$, where $a=(a_1,\ldots,a_n)$ and $a_i>0$ for
some $i$. We set $S_{i}=K^{*}$ for all $i$. As $\deg_{y_{i}}(G) <
q-1$ for all $i$, then $a_i<\left|S_{i}\right| =q-1$ for all $i$. Thus, by 
Lemma~ \ref{comb-null}, there are $x_{1},\ldots,x_{n}\in K^{*}$ so that 
$G\left( x_{1},\ldots ,x_{n}\right) \neq 0$, a contradiction to the
fact that $G$ vanishes on $\left( K^{\ast }\right) ^{n}$. 
\end{proof}

A polynomial of the form $t^a-t^b$,
with $a,b\in\mathbb{N}^s$, is called a {\it binomial}  
of $S$. An ideal generated 
by binomials is called a {\it binomial ideal\/}. 

\begin{lemma}\label{may1-1-11}
 Let $B=K[ t_{1},\ldots ,t_{s},y_{1},\ldots
 ,y_{n}] $ be a polynomial ring over an arbitrary field $K$. If
 $I'$ is a binomial ideal of $B$, then $I'\cap K[t_1,\ldots,t_s]$ is a
 binomial ideal.
\end{lemma}

\begin{proof} Let $S=K[t_1,\ldots,t_s]$ and let $\mathcal{G}$ be a
Gr\"obner basis of   
$I'$ with respect to the lexicographic order $y_1\succ\cdots\succ y_n\succ
 t_1\succ\cdots\succ t_s$. By Buchberger algorithm
\cite[Theorem~2, p.~89]{CLO} the set $\mathcal{G}$ consists of binomials and 
by elimination theory \cite[Theorem~2, p.~114]{CLO} the set
$\mathcal{G}\cap S$ is a Gr\"obner 
basis of $I'\cap S$. Hence $I'\cap S$ is a binomial ideal. See the
proof of \cite[Corollary~4.4, p.~32]{Stur1} for additional details.
\end{proof}

\begin{theorem}\label{22-06-10}
 Let $B=K[ t_{1},\ldots ,t_{s},y_{1},\ldots
 ,y_{n}] $ be a polynomial ring over a finite field
 $K$ with $q$ elements. Then 
$$I\left( X^{\ast }\right)=
    \left(t_{1}-y^{v_{1}},\ldots
    ,t_{s}-y^{v_{s}},y_{1}^{q-1}-1,\ldots,y_{n}^{q-1}-1\right) \cap S
$$
and $I(X^*)$ is a binomial ideal. 
\end{theorem}

\begin{proof} We set $I'=\left(t_{1}-y^{v_{1}},\ldots
    ,t_{s}-y^{v_{s}},y_{1}^{q-1}-1,\ldots,y_{n}^{q-1}-1\right)\subset
    B$.  
First we show the inclusion $I(X^*)\subset I'\cap S$. Take a
polynomial $F=F\left( t_{1},\ldots ,t_{s}\right)$ that vanishes on
$X^*$. We can write  
\begin{equation}\label{may1-11}
F=\lambda _{1}t^{m_{1}}+\cdots +\lambda _{r}t^{m_{r}}\ \  \left(
\lambda_{i}\in K^{\ast};\ m_{i}\in \mathbb{N}^{s}\right).
\end{equation}
Write
$m_i=(m_{i1},\ldots,m_{is})$ for $1\leq i\leq r$. Applying the binomial
theorem to expand the right hand side of the equality
$$
t_j^{m_{ij}}=\left[(t_j-y^{v_j})+y^{v_j}\right]^{m_{ij}},\ \ \ 
1\leq i\leq r,\ 1\leq j\leq s,
$$
we get the equality
$$
t_j^{m_{ij}}=
\left(\sum_{k=0}^{m_{ij}-1}
\binom{m_{ij}}{k}\left(t_j-y^{v_j})^{m_{ij}-k}(y^{v_j})^k
\right)\right)+(y^{v_j})^{m_{ij}}.
$$
As a result, we obtain
that $t^{m_i}$ can be written as:
$$
t^{m_i}=t_1^{m_{i1}}\cdots
t_s^{m_{is}}=p_i+(y^{v_1})^{m_{i1}}\cdots(y^{v_s})^{m_{is}},
$$
where $p_i$ is a polynomial in the ideal
$(t_1-y^{v_1},\ldots,t_s-y^{v_s})$. Thus, substituting
$t^{m_1},\ldots,t^{m_r}$ in Eq.~(\ref{may1-11}), we obtain  
that $F$ can 
be written as:
\begin{equation}\label{23-jul-10}
F=\sum_{i=1}^sg_i(t_i-y^{v_i})+F(y^{v_1},\ldots,y^{v_s})
\end{equation}
for some $g_1,\ldots,g_s$ in $B$. By the division algorithm
in $K[y_1,\ldots,y_n]$ (see \cite[Theorem~3, p.~63]{CLO}) we can
write 
\begin{equation}\label{23-jul-10-1}
F(y^{v_1},\ldots,y^{v_s})=\sum_{i=1}^nh_i(y_i^{q-1}-1)+G(y_1,\ldots,y_n)
\end{equation}
for some $h_1,\ldots,h_n$ in $K[y_1,\ldots,y_n]$, 
where the monomials that occur in $G=G(y_1,\ldots,y_n)$ are not divisible by 
any of the monomials $y_1^{q-1},\ldots,y_n^{q-1}$, i.e.,
$\deg_{y_i}(G)<q-1$ for $i=1,\ldots,n$.
Therefore, using Eqs.~(\ref{23-jul-10}) and (\ref{23-jul-10-1}), we
obtain the equality
\begin{equation}\label{23-jul-10-2}
F=\sum_{i=1}^sg_i(t_i-y^{v_i})+\sum_{i=1}^nh_i(y_i^{q-1}-1)+
G(y_1,\ldots,y_n).
\end{equation}
Thus to show that $F\in I'\cap S$ we need only show that $G=0$. We
claim that  
$G$ vanishes on $(K^*)^n$. Take an arbitrary sequence $x_1,\ldots,x_n$
of elements of $K^*$. Making $t_i=x^{v_i}$ for all $i$ in
Eq.~(\ref{23-jul-10-2}) and using that
$F$ vanishes on $X^*$, we obtain
\begin{equation}\label{23-jul-10-3}
0=F(x^{v_1},\ldots,x^{v_s})=\sum_{i=1}^sg_i'(x^{v_i}-y^{v_i})+
\sum_{i=1}^nh_i(y_i^{q-1}-1)+
G(y_1,\ldots,y_n),
\end{equation}
where $g_i'=g_i(x^{v_1},\ldots,x^{v_s},y_1,\ldots,y_n)$. Since
$(K^*,\,\cdot\, )$ is a group of order $q-1$,  
we can then make $y_i=x_i$ for all $i$ in
Eq.~(\ref{23-jul-10-3}) to get that $G$
vanishes on $(x_1,\ldots,x_n)$. This completes the proof of the 
claim. Therefore $G$ vanishes on $(K^*)^n$ and $\deg_{y_i}(G)<q-1$ 
for all $i$. Hence $G=0$ by Lemma~\ref{may1-2-11}. 

Next we show the inclusion $I(X^*)\supset I'\cap S$. Take a
polynomial $f$ in $I'\cap S$. Then we can
write
\begin{equation}\label{sept1-09}
f=\sum_{i=1}^sg_i(t_i-y^{v_i})+\sum_{i=1}^nh_i(y_i^{q-1}-1)
\end{equation}
for some polynomials $g_1,\ldots,g_s, h_1,\ldots,h_n$ in $B$. Take a point 
 $P=(x^{v_1},\ldots,x^{v_s})$ in $X^*$. Making $t_i=x^{v_i}$ in
Eq.~(\ref{sept1-09}), we get 
$$
f(x^{v_1},\ldots,x^{v_s})=\sum_{i=1}^sg_i'(x^{v_i}-y^{v_i})+
\sum_{i=1}^nh_i'(y_i^{q-1}-1),
$$
where $g_i'=g_i(x^{v_1},\ldots,x^{v_s},y_1,\ldots,y_n)$ and 
$h_i'=h_i(x^{v_1},\ldots,x^{v_s},y_1,\ldots,y_n)$. 
Hence making $y_i=x_i$ for all $i$, we get that 
$f(P)=0$. Thus $f$ vanishes on $X^*$.
\end{proof}

In this paper we are always working over a finite field $K$. If
$K=\mathbb{C}$ is the field of complex numbers and $X$ is
an affine toric variety, i.e., 
$$X=V(\mathfrak{p})=\{P\in K^n\vert\,
f(P)=0 \mbox{ for all }f\in\mathfrak{p} \}$$ 
is the {\it zero set\/} of a toric
ideal $\mathfrak{p}$, then by the Nullstellensatz \cite[Theorem~1.6]{Eisen} we
have that $I(X)=\mathfrak{p}$. This means that $I(X)$ is a binomial ideal. For
infinite fields, we can use the Combinatorial Nullstellensatz (see
Theorem~\ref{comb-null}) to show the following description 
of $I(X^*)$. We refer to \cite{Stur1} for the theory of toric 
ideals. 

\begin{proposition}\label{infinite-field}
Let $B=K[t_1,\ldots,t_s,y_1,\ldots,y_n]$ be a polynomial ring
over an infinite field $K$. Then
$$
I(X^*)=(t_1-y^{v_1},\ldots,t_s-y^{v_s})\cap S
$$
and $I(X^*)$ is the toric ideal of $K[y^{v_1},\ldots,y^{v_s}]$. 
\end{proposition}

Our next aim is to show how to compute $I(Y)$. For $f\in S$ of degree $e$ define 
$$
f^h=u^ef\left({t_1}/{u},\ldots,{t_s}/{u}\right),
$$
that is,  $f^h$ is the {\it homogenization\/} of the polynomial 
$f$ with respect to $u$. The {\it homogenization\/} 
of $I(X^*)\subset S$ is the ideal $I(X^*)^h$ of $S[u]$ 
given by 
$$
I(X^*)^h=(\{f^h|\, f\in I(X^*)\}).
$$ 

Let $\succ$ be the {\it elimination order\/} on the 
monomials of $S[u]$ with respect to 
$t_1,\ldots,t_s,t_{s+1}$, where $u=t_{s+1}$. Recall that 
this order is defined as  $t^b\succ t^a$
if and only if the total degree of $t^b$
in  the variables $t_1,\ldots,t_{s+1}$ is greater than that of $t^a$, or
both degrees are equal, and the last nonzero component 
of $b-a$ is negative. 

\begin{definition}\label{projective-closure-def} The 
{\it projective closure\/} of $X^*$, denoted by $\overline{X^*}$, is
given by $\overline{X^*}:=\overline{Y}$, where 
$\overline{Y}$ is the closure of $Y$ in the Zariski topology of $\mathbb{P}^s$.
\end{definition}

\begin{lemma}\label{elim-ord-hhomog}
If $f_1,\ldots,f_r$ is a Gr\"obner basis of $I(X^*)$, then
$f_1^h,\ldots,f_r^h$ form a Gr\"obner basis and the following 
equalities hold{\rm:} 
$$
I(Y)=I(X^*)^h=(f_1^h,\ldots,f_r^h).
$$
\end{lemma}

\begin{proof} In our situation $\overline{X^*}=\overline{Y}=Y$ because $K$ is a
finite field. Thus, the result follows readily from
\cite[Propositions~2.4.26 and 2.4.30]{monalg}.  
\end{proof}

\begin{corollary}\label{may4-11} The dimension and the length of $C_{X^*}(d)$ can be
computed using Gr\"obner basis.  
\end{corollary}

\begin{proof} By Lemma~\ref{elim-ord-hhomog} we can find a generating
set of $I(Y)$ using Gr\"obner basis. Thus, using the computer algebra
system {\em Macaulay\/}$2$ \cite{computations-macaula2,mac2}, we can compute the Hilbert
function and the degree of $S[u]/I(Y)$, i.e., we can compute the  
dimension and the length of $C_Y(d)$. Consequently, 
Theorem~\ref{bridge-affine-projective} allows to compute the dimension
and the length of $C_{X^*}(d)$ using Gr\"obner basis. 
\end{proof}

Putting the results of this section together we obtain the following
procedure.

\begin{procedure}\label{mac-proc-affine-codes} 
The following simple procedure for {\it
Macaulay\/}$2$ computes the dimension and the length of a
parameterized affine code $C_{X^*}(d)$ of degree $d$. 
\begin{verbatim}
R=GF(q)[y1,...,yn,t1,...,ts,u,MonomialOrder=>Eliminate n]
I'=ideal(t1-y1^{v_1},...,t_s-y^{s},y1^{q-1}-1,...,yn^{q-1}-1)
I(X^*)=ideal selectInSubring(1,gens gb I')
I(Y)'=homogenize(I(X^*),u)
S=GF(q)[t1,...,ts,u]
I(Y)=substitute(I(Y)',S)
degree I(Y)
hilbertFunction(d,I(Y))
\end{verbatim}
\end{procedure}

\begin{example}\label{k3-affine-parameters} Let $X^*$ be the affine algebraic toric set parameterized by
$y_1y_2,y_2y_3,y_1y_3$ and let $C_{X^*}(d)$ be its parameterized
affine code of order $d$ over the field
$K=\mathbb{F}_5$. Using {\em Macaulay\/}$2$, together with
Procedure~\ref{mac-proc-affine-codes}, we obtain:
\begin{eqnarray*}
I(X^*)&=&(t_3^4-1,t_2^2t_3^2-t_1^2,t_1^2t_3^2-t_2^2,t_2^4-1,t_1^2t_2^2-t_3^2,t_1^4-1),\\ 
I(Y)&=&(t_3^4-t_4^4,t_2^2t_3^2-t_1^2t_4^2,t_1^2t_3^2-t_2^2t_4^2,t_2^4-t_4^4,t_1^2t_2
      ^2-t_3^2t_4^2,t_1^4-t_4^4),
\end{eqnarray*}
\begin{eqnarray*}
&&\left.
\begin{array}{c|c|c|c|c|c}
 d & 1 & 2 & 3 &4 &5\\ \hline
|X^*|& 32 & 32 & 32&32&32 \\ \hline
\dim C_{X^*}(d)& 4 & 10 & 20&29&32  \\ \hline
\delta_{X^*}(d) & 23 &8 &  & & 1
\end{array}\right.
\end{eqnarray*}
The minimum distance was also computed with {\it Macaulay\/}$2$. 
\end{example}

\medskip

\begin{center}
ACKNOWLEDGMENT
\end{center}

\noindent We thank the referee for the careful reading of the paper
and for the improvements that he/she suggested.

\bibliographystyle{plain}

\end{document}